\documentclass[11pt]{amsart}
\usepackage{amssymb,amsmath,amsfonts,amscd,euscript}

\newcommand{\nc}{\newcommand}

\numberwithin{equation}{section}
\newtheorem{thm}{Theorem}[section]
\newtheorem{prop}[thm]{Proposition}
\newtheorem{lem}[thm]{Lemma}
\newtheorem{cor}[thm]{Corollary}
\theoremstyle{remark}
\newtheorem{rem}[thm]{Remark}

\newtheorem{example}[thm]{Example}
\newtheorem{dfn}[thm]{Definition}

\nc{\gl}{\mathfrak{gl}}
\nc{\GL}{\mathfrak{GL}}
\nc{\g}{\mathfrak{g}}
\nc{\gh}{\widehat\g}
\nc{\h}{\mathfrak{h}}
\nc{\la}{\lambda}
\nc{\al}{\alpha }
\nc{\be}{\beta }
\nc{\ve}{\varepsilon }
\nc{\om}{\omega }

\nc{\ta}{\theta}
\nc{\veps}{\varepsilon}
\nc{\ch}{{\mathop {\rm ch}}}
\nc{\Tr}{{\mathop {\rm Tr}\,}}
\nc{\Id}{{\mathop {\rm Id}}}
\nc{\ad}{{\mathop {\rm ad}}}
\nc{\bra}{\langle}
\nc{\ket}{\rangle}
\nc{\x}{{\bf x}}
\nc{\bs}{{\bf s}}
\nc{\bp}{{\bf p}}
\nc{\bc}{{\bf c}}
\nc{\pa}{\partial}
\nc{\ld}{\ldots}
\nc{\cd}{\cdots}
\nc{\hk}{\hookrightarrow}
\nc{\T}{\otimes}
\newcommand{\bea}{\begin{equation}}
\newcommand{\ena}{\end{equation}}
\nc{\gr}{\mathrm{gr}}
\nc{\ov}{\overline}

\nc{\cO}{\mathcal O}
\nc{\msl}{\mathfrak{sl}}
\nc{\mgl}{\mathfrak{gl}}
\nc{\U}{\mathrm U}
\nc{\V}{\EuScript V}
\nc{\bH}{\EuScript H}
\nc{\Res}{\mathrm{Res\ }}

\newcommand{\bC}{{\mathbb C}}
\newcommand{\bZ}{{\mathbb Z}}

\newcommand{\bP}{{\mathbb P}}
\newcommand{\bG}{{\mathbb G}}

\newcommand{\fb}{{\mathfrak b}}

\newcommand{\fn}{{\mathfrak n}}

\newcommand{\Fl}{\EuScript{F}}

\begin{document}

\title[$\bG_a^M$ degeneration of flag varieties]
{$\bG_a^M$ degeneration of flag varieties}

\author{Evgeny Feigin}
\address{Evgeny Feigin:\newline
Tamm Theory Division,
Lebedev Physics Institute,\newline
Leninisky prospect, 53,
119991, Moscow, Russia,\newline
{\it and }\newline
French-Russian Poncelet Laboratory, Independent University of Moscow
}
\email{evgfeig@gmail.com}

\begin{abstract}
Let $\Fl_\la$ be a generalized flag variety of a simple Lie group $G$ embedded into 
the projectivization of an irreducible $G$-module $V_\la$. We define a flat 
degeneration 
$\Fl_\la^a$, which is a $\bG^M_a$ variety.
Moreover, there exists a larger group $G^a$ acting on $\Fl_\la^a$, which is a degeneration 
of the group $G$. The group $G^a$ contains $\bG^M_a$ as a normal subgroup.
If $G$ is of type $A$, then the degenerate flag
varieties can be embedded into the product of Grassmanians and thus to the product 
of projective spaces. The defining ideal
of $\Fl^a_\la$ is generated by the set  of degenerate Pl\" ucker relations.
We prove that the coordinate ring of $\Fl_\la^a$ is isomorphic to a direct 
sum of dual PBW-graded $\g$-modules.
We also prove that there exist bases in multi-homogeneous components of the coordinate 
rings, parametrized by the semistandard PBW-tableux, which are analogues of 
semistandard tableux.
\end{abstract}

\maketitle

\section*{Introduction}
Let $G$ be a simple Lie group with the Lie algebra $\g$. The flag varieties 
associated with $G$ are the quotients $G/P$ by the parabolic subgroups
(see e.g.\cite{Kum1}). These varieties can be realized as  $G$-orbits 
$G [v_\la]\hk\bP(V_\la)$ in the projectivization of irreducible
$G$-modules $V_\la$ with highest weight vector $v_\la$, $[v_\la]=\bC v_\la$. 
We denote $\Fl_\la=G[v_\la]$.  

In this paper we introduce a new family of varieties $\Fl^a_\la$, which 
are flat degenerations of $\Fl_\la$ (the superscript $a$ is for abelian,
see the explanations below). 
We note that in the literature there exists  a degeneration of flag varieties into 
toric varieties (see 
\cite{C}, \cite{GL}, \cite{L}). Our degenerate flags are $\bG^M_a$ varieties, i.e.
they  are equipped with an 
action of the $M$-fold product $\bG^M_a$ of the additive group  of the field
with an open orbit, where $M$ is the dimension of a maximal unipotent subgroup of $G$.

The varieties $\Fl_\la^a$ are defined as follows. 
Let $F_s$, $s\ge 0$ be the PBW filtration on $V_\la$ (see \cite{Fe1}, \cite{Fe2}, \cite{FFoL},
\cite{Kum2}) 
\[
F_s=\mathrm{span}\{x_1\dots x_l v_\la:\ x_i\in\g, l\le s \}.
\]   
We define $V_\la^a=F_0\oplus_{s\ge 0} F_{s+1}/F_s$. 
Let $\g=\fn\oplus\h\oplus \fn^-$ be the Cartan decomposition. 
The space $V_\la^a$ has a natural structure of a module over the degenerate algebra
$\g^a$. The algebra $\g^a$ is isomorphic to $\g$ as a vector space. It is
a direct sum of two subalgebras: one is the Borel subalgebra
$\fb=\fn\oplus\h$ and the second is an abelian ideal $(\fn^-)^a$, which is
an abelian subalgebra isomorphic to $\fn^-$ as a vector space.
The corresponding Lie group $G^a$ is a semidirect product
$\bG_a^M\rtimes B$ of the normal
subgroup $\bG_a^M$ ($M=\dim\fn$) and the Borel subgroup $B$. 
%We note that $V_\la^a=S^\bullet((\fn^-)^a)v_\la$. 
We define the varieties
$\Fl_\la^a$ as the closure of the $\bG^M_a$-orbit of the highest weight vector:
\[
\Fl_\la=\overline{\bG^M [v_\la]}\hk \bP(V_\la^a).
\]   
We note that an affine space $\bG^M_a [v_\la]$ does not coincide with the closure
$\overline{\bG_a^M [v_\la]}$, but form a dense  open $\bG^M_a$-orbit.  
Hence for all $\la$ the varieties $\Fl_\la$ are the so-called $\bG_a^M$ varieties
\cite{HT}, \cite{AS}, \cite{A}. 

We study the varieties $\Fl^a_\la$ in the case $\g=\msl_n$. 
In this case the varieties $\Fl_\la$ are isomorphic to partial flag varieties.
In particular, the ones corresponding to fundamental weights are 
Grassmanians $Gr(d,n)$. There exist embeddings of partial flags into the product
of projective spaces and the image is given by Pl\"ucker relations. These 
relations describe coordinate rings of flag varieties (see \cite{Fu}).
Since all fundamental weights are cominuscule in the $\msl_n$ case,
we have $\Fl^a_{\om_d}\simeq \Fl_{\om_d}\simeq Gr(d,n)$.
The results of \cite{FFoL} imply that each degenerate flag variety can be
embedded into the product of Grassmanians and thus into the product of projective spaces.
We show that this embedding can be described in terms of the explicit set of 
multi-homogeneous algebraic equations, which are obtained from the Pl\"ucker relations
by certain degeneration. We prove that the degeneration $\Fl\to \Fl^a$ is flat.  
We also show that (as in the classical case)
the coordinate ring of $\Fl^a_\la$ is isomorphic to a direct sum of certain dual
modules $(V_\mu^a)^*$.

Our main technical tool is combinatorics of the semistandard 
PBW-tableaux. Recall that the usual semistandard tableaux serve
as labeling set for  bases of the coordinate rings of flag varieties. We construct
another set (which we call the set of semistandard PBW-tableux) of Young diagrams filled 
by numbers such that these tableaux provide a basis in the coordinate 
rings of $\Fl_\la^a$ (and of $\Fl_\la$ as well). 

Finally, we note that the case $\g=\msl_n$ is special, since all fundamental weights
are cominuscule. This is important for our construction, since the operators 
from  the  nilpotent radical act as pairwise commuting operators in each
module $V_{\om_d}$ even before passing to $V^a_{\om_d}$. This gives an identification
$\Fl^a_{\om_d}\simeq \Fl_{\om_d}\simeq Gr(d,n)$ and simplifies the whole picture
(see \cite{A}).
For other algebras there exist Grassmanians, which have to be degenerated.
It is very interesting to study these degenerations as well as arbitrary 
degenerations $\Fl^a_\la$ for general $\g$.

Our paper is organized as follows:\\
\noindent In Section $1$ we recall notations and main facts about Lie algebras
and flag varieties in type $A$.\\
\noindent In Section $2$ we introduce the $\bG_a^M$ degeneration of
the generalized flag varieties for arbitrary simple Lie group $G$.\\
\noindent In Section $3$ we state our main results and provide examples.\\
\noindent In Sections $4$ and $5$ we prove the statements from Section $3$.
Section $4$ addresses the combinatorial questions and in Section $5$ we
prove algebro-geometric results.

\section{The classical case}
\subsection{Notations}
Let $\g$ be a simple Lie algebra with the Cartan decomposition
$\g=\fn\oplus\h\oplus \fn^-$ and the Borel subalgebra $\fb=\fn\oplus\h$.
Let $\om_i$, $\al_i$, $i=1,\dots,l$ be fundamental weights and simple roots of $\g$.
Here $l=\dim\h= \mathrm{rk}\ \g$. 
Let $Q$ and $P$ be root and weight lattices in $\h^*$:
$P$ is generated by $\omega_i$ and $Q$ is generated by $\al_i$. 
We set
\[
Q_\pm=\bigoplus_{i=1}^l \pm\bZ_{\ge 0} \al_i,\ P_\pm=\bigoplus_{i=1}^l \pm\bZ_{\ge 0} \omega_i,\
\bZ_{\ge 0}=\{n\in\bZ:\ n\ge 0\}.
\]
We denote by $(\cdot,\cdot)$ the Killing form on $\h^*$. In particular, we have
$(\al_i,\omega_j)=\delta_{i,j}$.

Consider the weight decomposition
\[
\fn=\bigoplus_{\al\in Q_+} \fn_\al,\ \fn^-=\bigoplus_{\al\in Q_+} \fn^-_{-\al},
\]
where $\fn_\al$ and $\fn^-_{-\al}$ are one-dimensional spaces spanned by elements
$e_\al$ and $f_\al$. One has
\[
[h,e_\al]=\al(h)e_\al, \ [h,f_\al]=-\al(h)f_\al,\ h\in\h.
\]
We denote the elements $e_{\al_i}$ by $e_i$ and $f_{\al_i}$ by $f_i$. Then
$e_\al$, $f_\al$, $\al\in Q_+$ and $h_i=[e_i,f_i]$ form the Chevalley basis of $\g$.

For any $\la=\sum_{i=1}^l m_i\omega_i\in P_+$ let $V_\la$ be irreducible highest weight
$\g$-module with highest weight $\la$. Let $v_\la\in V_\la$ be a highest weight vector.
Then one has
\[
hv_\la=\la(h)v_\la \ \forall h\in\h,\ \ \fn v_\la=0,\ \mathrm{U}(\fn^-) v_\la=V_\la.
\]

Let $G$ be a simple Lie group with the Lie algebra $\g$. Let $B$, $H$, $N$ and $N^-$
be the Borel, Cartan and unipotent subgroups of $G$. The corresponding
Lie algebras are $\fb$, $\h$, $\fn$ and $\fn^-$.
Each space $V_\la$, $\la\in P_+$ is equipped with the natural structure
of $G$-module. Therefore $G$ acts on the projectivization
$\bP(V_\la)$. The (generalized) flag variety $\Fl_\la\hk \bP(V_\la)$ is defined
as the $G$-orbit of the line $\bC v_\la$. Each variety $\Fl_\la$ is isomorphic
to  the quotient of $G$ by the parabolic subgroup leaving the point $\bC v_\la\in\bP(V_\la)$
invariant.

\subsection{Type $A$ case: Pl\"ucker relations, coordinate rings and semi\-stan\-dard tableaux}
\label{clas}
In this subsection we recall the main ingredients of the theory of 
flag varieties in the case $\g=\msl_n$. Our main reference is \cite{Fu} (see \cite{Kum1}
for more Lie-theoretic approach).

Let $1\le d_1<\dots <d_s\le n-1$ be a sequence of increasing numbers. Then for any positive
integers $a_1,\dots,a_s$ the variety $\Fl_{a_1\om_{d_1}+\dots + a_s\om_{d_s}}$ is isomorphic to
the partial flag variety
\[
\Fl(d_1,\dots,d_s)=\{V_1\hk V_2\hk\dots\hk V_s\hk \bC^n:\ \dim V_i=d_i\}.
\]
In particular, if $s=1$, then $\Fl(d)$ is the Grassmanian $Gr(d,n)$ and for $s=n-1$
$\Fl(1,\dots,n-1)$ is the variety of full flags. We recall that 
\[
V_{\om_d}=\Lambda^d(V_{\om_1})=\Lambda^d(\bC^n)
\] 
and the embedding $Gr(d,n)\hk \bP(\Lambda^d V_{\om_1})$ is defined as follows: 
a subspace with a basis $w_1,\dots,w_d$ maps to $\bC w_1\wedge\dots\wedge w_d$.
For general sequence $d_1,\dots,d_s$ one has embeddings:
\[
\Fl(d_1,\dots,d_s)\hk Gr(d_1,n)\times\dots\times Gr(d_s,n)\hk
\bP(V_{\om_{d_1}})\times\dots\times \bP(V_{\om_{d_s}}).
\]  
The composition of these embeddings is called the Pl\"ucker embedding. 
The image is described
explicitly in terms of Pl\"ucker relations. Namely, let $v_1,\dots,v_n$ be a 
basis of $\bC^n=V_{\om_1}$. Then one gets  a basis $v_J$ of $V_{\om_d}$ 
$v_J=v_{j_1}\wedge\dots\wedge v_{j_d}$  labeled by sequences 
$J=(1\le j_1<j_2<\dots <j_d\le n)$. Let $X_J\in V_{\om_d}^*$ be the dual basis.
We denote by the same symbols the coordinates of a vector $v\in V_{\om_d}$:
$X_J=X_J(v)$. 
The image of the embedding 
\[
\Fl(d_1,\dots,d_s)\hk\times_{i=1}^s \bP(V_{\om_{d_i}})
\]
is defined by the Pl\"ucker relations. These relations are labeled by a pair of 
numbers $p\ge q$, $p,q\in\{d_1,\dots,d_s\}$, by a number $k$, $1\le k\le q$ and by a pair 
of sequences $L=(l_1,\dots,l_p)$, $J=(j_1,\dots,j_q)$, $1\le l_\al,j_\beta\le n$. The
corresponding relation is denoted by $R^k_{L,J}$ and is given by
\begin{equation}\label{PR}
R^k_{L,J}=X_LX_J - \sum_{1\le r_1 <\dots <r_k\le p} X_{L'}X_{J'},
\end{equation} 
where $L'$,$J'$ are obtained from $L$, $J$ by interchanging $k$-tuples 
$(l_{r_1},\dots,l_{r_k})$ and $(j_1,\dots,j_k)$ in $L$ and $J$ respectively, i.e.
\begin{gather*}
J'=(l_{r_1},\dots,l_{r_k},j_{k+1},\dots,j_q),\\ 
L'=(l_1,\dots,l_{r_1-1},j_1,l_{r_1+1},\dots,l_{r_2-1},j_2,\dots,l_p).
\end{gather*}
We note that for any $\sigma\in S_d$ the equality
\[
X_{j_{\sigma(1)},\dots,j_{\sigma(d)}}=(-1)^\sigma X_{j_1,\dots,j_d}
\]
is assumed in \eqref{PR}. We denote the ideal generated by all $R^k_{L,J}$ by
$I(d_1,\dots,d_s)$. An important fact is that this ideal is prime.

Recall that the coordinate ring of $\Fl(d_1,\dots,d_s)$ is the quotient ring
\[
Q(d_1,\dots,d_s)=\bC[X_{j_1,\dots,j_d}]/I(d_1,\dots,d_s),
\] 
where the variables $X_{j_1,\dots,j_d}$ are labeled by 
$d=d_1,\dots,d_s$ and $1\le j_1<\dots <j_d\le n$.

The ring $Q(d_1,\dots,d_s)$ has two other descriptions. First, recall that for any $\la,\mu\in P^+$
there exists an  embedding of $\g$-modules 
$V_{\la+\mu}\hk V_\la\T V_\mu$, $v_{\la+\mu}\mapsto v_\la\T v_\mu$. 
Therefore we have dual surjective maps $V_\la^*\T V_\mu^*\to V_{\la+\mu}^*$,
which define an algebra
\[
\bar Q(d_1,\dots,d_s)=\bigoplus_{\la=m_1\om_{d_1}+\dots +m_s\om_{d_s}} V_\la^*.
\]
Then one has an isomorphism of algebras $\bar Q(d_1,\dots,d_s)\simeq Q(d_1,\dots,d_s)$.

One can also embed the ring $Q(d_1,\dots,d_s)$ into a polynomial ring 
(thus proving that $I(d_1,\dots,d_s)$ is prime).
Namely, let $Z_{i,j}$, $1\le i,j\le n$ be a set of variables.
Then the map 
\begin{equation}\label{polyemb}
X_{j_1,\dots,j_d}\mapsto \det (Z_{l,j_l})_{1\le l\le d}
\end{equation}
induces the embedding $Q(d_1,\dots,d_s)\hk\bC[Z_{i,j}]_{1\le i,j\le n}$.
 
The ring $Q(d_1,\dots,d_s)$ has a basis described in terms of the semistandard 
tableaux. Namely, for a partition $\la=(\la_1\ge\dots\ge \la_{n-1}\ge 0)$ we denote by $Y_\la$ 
the corresponding Young diagram. For example, $Y_{(7,5,4,2,2)}$ is the following 
diagram:
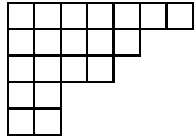
\[
\begin{picture}(100,50)
\put(0,0){\line(0,1){50}}
\put(10,0){\line(0,1){50}}
\put(20,0){\line(0,1){50}}
\put(20,20){\line(0,1){30}}
\put(30,20){\line(0,1){30}}
\put(40,20){\line(0,1){30}}
\put(50,30){\line(0,1){20}}
\put(60,40){\line(0,1){10}}
\put(70,40){\line(0,1){10}}

\put(0,0){\line(1,0){20}}
\put(0,10){\line(1,0){20}}
\put(0,20){\line(1,0){40}}
\put(0,30){\line(1,0){50}}
\put(0,40){\line(1,0){70}}
\put(0,50){\line(1,0){70}}
\end{picture}
\]

We number the rows and columns of $Y_\la$ from up to down and from left
to right. Thus, the boxes of $Y_\la$ are labeled by pairs $(i,j)$. 
We denote by $\mu_j$ the length of the $j$-th column. 
  
A tableau $T$ of shape $\la$ is a filling of $Y_\la$ with numbers $T_{i,j}\in\{1,\dots,n\}$.
The number $T_{i,j}$ is attached to the box in the $i$-th row and $j$-th column.
The numbers $T_{i,j}$ are subject to the condition: $i_1<i_2$ implies
$T_{i_1,j}<T_{i_2,j}$. A tableau is called semistandard if in addition
$j_1<j_2$ implies $T_{i,j_1}\le T_{i,j_2}$. 

Let $P^+(d_1,\dots,d_s)=\bZ_+\om_{d_1}\oplus\dots\oplus \bZ_+\om_{d_s}$. 
Then we have a natural decomposition
\[
Q(d_1,\dots,d_s)=\bigoplus_{\la\in P^+(d_1,\dots,d_s)} Q_\la(d_1,\dots,d_s)=
\bigoplus_{\la\in P^+(d_1,\dots,d_s)} V_\la^*.
\]     
Recall that to a weight $\la=\sum_{i=1}^{n-1} m_i\om_i$ one can attach a partition
\[
(m_1+\dots +m_{n-1}, m_2+\dots +m_{n-1},\dots,m_{n-1}),
\]   
which we denote by the same symbol $\la$. 
 
Finally, for a  tableau $T$ of shape $\la$ we define an element 
$$X_T=\prod_{j=1}^{\la_1} X_{T_{1,j},\dots,T_{\mu_j,j}}$$
(note that $\la\in P^+(d_1,\dots,d_s)$ implies $\mu_j\in\{d_1,\dots,d_s\}$ for all $j$).
Then the elements $X_T$ labeled by the shape $\la$ semistandard tableaux form a basis 
of $Q(d_1,\dots,d_s)$.

\section{Abelian degenerations}\label{abelian}
We first recall the definition of the PBW-filtration on a highest weight 
module $V_\la$. 

The space $V_\la$ is equipped with the increasing PBW filtration $F_s$ defined as follows:
\[
F_0=\bC v_\la,\ F_s=\mathrm{span}\{x_1\dots x_kv_\la:\ k\le s, x_i\in\fn^- \}.
\]
In other words, $F_{s+1}=F_s + \fn^- F_s$.
We denote by $V_\la^a$ the associated graded space,
\begin{equation}\label{PBW}
V_\la^a=F_0\oplus\bigoplus_{s=1}^\infty F_s/F_{s-1}.
\end{equation}
The superscript $a$ stands for the abelian, see the explanation below.
In what follows we write $V_\la^a=\bigoplus_{s\ge 0} V_\la^a(s)$,
$V_\la^a(s)=F_s/F_{s-1}$. An element $x\in V_\la^a(s)$ is said to be homogeneous
of the PBW degree $s$.

Let us describe the Lie algebra and Lie group acting on $V_\la^a$.
\begin{dfn}\label{ga}
The Lie algebra $\g^a$ is isomorphic to $\g$ as a vector space and the bracket
$[\cdot,\cdot]^a$ is given by the formulas
\begin{gather*}
[h,f_\al]^a=-\al(h)f_\al,\ [h,e_\al]^a=\al(h)e_\al,\ [h,h_1]^a=0\qquad \forall h,h_1\in \h, \al\in Q_+,\\
[f_\al,f_\beta]^a=0,\ [e_\al,e_\beta]^a=[e_\al,e_\beta] \qquad \forall \al,\be\in Q_+,\\
[e_\al,f_\be]^a=\begin{cases}
[e_\al,f_\be], \text{ if } \be - \al\in Q_+,\\
0, \text { otherwise}.
\end{cases}
\end{gather*}
\end{dfn}
It is straightforward to check that  the bracket $[\cdot,\cdot]^a$
satisfies the Jacobi identity.
The superscript $a$ stays for the abelian, since the subalgebra $\fn^-$ is abelian with respect
to the bracket $[\cdot,\cdot]^a$.
In what follows we omit the superscript $a$ in the bracket $[\cdot,\cdot]^a$ if it is clear
what algebra $\g$ or $\g^a$ is considered. 

\begin{rem}
The Lie algebra $\g^a$ is a semidirect sum of the Borel subalgebra $\fb$
and an abelian ideal $(\fn^-)^a$, which is isomorphic to $\fn^-$ as a vector space.
The action of $\fb$ on  $(\fn^-)^a$ can be described as follows. The Lie algebra
$\g$ is naturally equipped with a structure of adjoint $\fb$-module and $\fb\hk\g$
is a submodule. Therefore we obtain a structure of quotient $\fb$-module on the space
$(\fn^-)^a\simeq \g/\fb$.
\end{rem}

\begin{rem}
The Lie algebra $\g^a$ can be considered as a degeneration of $\g$. In fact,
let $c^{++}_{\al,\be}$, $c^{--}_{\al,\be}$ and $c^{+-}_{\al,\be}$ be the structure constants
of $\g$, i.e.
\begin{gather*}
[e_\al,e_\be]=c^{++}_{\al,\be}e_{\al+\be},\
[f_\al,f_\be]=c^{--}_{\al,\be}f_{\al+\be},\\
[e_\al,f_\be]=c^{+-}_{\al,\be}f_{\be-\al},\ \be-\al\in Q_+,\
[e_\al,f_\be]=c^{+-}_{\al,\be}e_{\al - \be},\ \al-\be\in Q_+,\\
[e_\al,f_\al]=h_\al\in \h.
\end{gather*}
Let $\g(\veps)$ be the subalgebra of $\g$ spanned by the elements
$e_\al(\veps)=e_\al$,  $f_\al(\veps)=\veps f_\al$, $\al\in Q_+$ and $\h$. Obviously,
$\g(\veps)=\g$ for any $\veps\ne 0$. Then one has
\begin{gather*}
[e_\al(\veps),e_\be(\veps)]=c^{++}_{\al,\be}e_{\al+\be}(\veps),\
[f_\al(\veps),f_\be(\veps)]=\veps c^{--}_{\al,\be}f_{\al+\be}(\veps),\\
[e_\al(\veps),f_\be(\veps)]=c^{+-}_{\al,\be}f_{\be-\al}(\veps),\ \be-\al\in Q_+,\\
[e_\al(\veps),f_\be(\veps)]=\ve c^{+-}_{\al,\be}e_{\al - \be}(\veps),\ \al-\be\in Q_+,\\
[e_\al(\veps),f_\al(\veps)]=\veps h_\al.
\end{gather*}
In the limit $\veps\to 0$, the commutation relations above give the relations from Definition \ref{ga}.
\end{rem}

We now define the Lie group $G^a$, which corresponds to $\g^a$. 
Let $M=\dim \fn$ and let ${\mathbb G}_a$ be the additive group of the field $\bC$. The Lie group
$\bG_a^M$ can be naturally identified with the Lie group of the abelian Lie algebra
$(\fn^-)^a\hk \g^a$. We note that $B$ acts on $\g$ via the restiction of the adjoint $G$-action.
Since $\fb\hk\g$ is $B$-invariant, we obtain a structure of $B$-module on the vector space
$(\fn^-)^a\simeq \g/\fb$.
\begin{dfn}
The Lie group $G^a$ is a semidirect product $\bG_a^M\rtimes B$ of the normal
subgroup $\bG_a^M$ and the Borel subgroup $B$. The action by conjugation of $B$ on $\bG_a^M$ is induced
from the $B$-action on $(\fn^-)^a$ as above.
\end{dfn}

The following proposition is simple, but of the key importance for us.
Recall the graded modules $V^a_\la=F_0\oplus\bigoplus_{s\ge 1} F_s/F_{s-1}$ (see \eqref{PBW}).
\begin{prop}
For any $\la\in P_+$ the structure of $\g$-module on $V_\la$ induces the structures of
$\g^a$- and $G^a$-modules on $V_\la^a$.
\end{prop}
\begin{proof}
The action of the Lie algebra $\fn^-$ on $V_\la$ induces the action of the abelian algebra
$(\fn^-)^a\hk \g^a$ on $V_\la^a$. In fact, for any $\al\in Q_+$ we have $f_\al F_s\hk F_{s+1}$.
Therefore, we have the action of operators $f_\al$ on the graded space $V_\la^a$, mapping
$F_s/F_{s-1}$ to $F_{s+1}/F_s$. From definition of $F_s$ we conclude that all such operators commute.
Now consider the action of $\fb$ on $V_\la$. Since $xF_s\hk F_s$ for any $x\in\fb$, we obtain
the action of $\fb$ on the graded space $V_\la^a$ mapping each quotient $F_s/F_{s-1}$
to itself. Now it is easy to see that the action of abelian
$(\fn^-)^a$ and of $\fb$ satisfy the relations from Definition \ref{ga}. Thus we obtain a structure
of $\g^a$-module on $V_\la^a$. The group $G^a$ acts on $V_\la^a$ by the exponents of the operators
from $\g^a$.
\end{proof}

\begin{rem}
The space $V_\la^a$ is a cyclic module over the universal enveloping algebra of $(\fn^-)^a\hk\g^a$
and therefore can be identified with a quotient of the polynomial algebra in variables
$f_\al$, $\al\in Q_+$ modulo some ideal $I(\la)$: 
$V_\la^a\simeq \bC[f_\al]_{\al\in Q_+}/I(\la)$. The operators $e_\al\in\fn$ act on
this quotient as differential operators. This ideal in type $A$ was computed in \cite{FFoL}.
\end{rem}

\begin{example}
Let $\g=\msl_n$, $\la=\om_d$, $1\le d\le n-1$. Recall the basis $v_J$ 
of the module $V_{\om_d}=\Lambda^d(\bC^n)$. For a sequence $J=(j_1,\dots,j_d)$
we define the PBW degree of $J$ by the formula
\begin{equation}\label{PBWdegree}
\deg J=\#\{r:\ j_r>d\}.
\end{equation}  
Then the $s$-th space of the PBW filtration $F_s$ is spanned by the vectors
$v_J$ with $\deg J\le s$. Therefore, the images of the elements $v_J$ with
$\deg J=s$ define a basis of $F_s/F_{s-1}\hk V_{\om_d}^a$. 
We denote these images by the symbol $v^a_J$ and the coordinates (dual basis)
by $X^a_J$.

The action of $\g^a$ on $V_{\om_d}^a$ can be written explicitly. Namely,
for $1\le i\le j<n$ let $\al_{i,j}=\al_i+\dots +\al_j$ be positive roots of 
$\msl_n$ and let $f_{i,j}=f_{\al_{i,j}}$, $e_{i,j}=e_{\al_{i,j}}$. Then it is easy to see
that
\begin{equation}\label{action}
f_{i,j} v^a_J=
\begin{cases}
\sum_{r=1}^d \delta_{i,j_r} v^a_{j_1,\dots,j_{r-1},j+1,j_{r+1},\dots,j_d}, 
\text{ if } i\le d, j\ge d,\\
0, \text{ otherwise };
\end{cases}
\end{equation}
\begin{equation*}
e_{i,j} v^a_J=
\begin{cases}
\sum_{r=1}^d \delta_{j+1,j_r} v^a_{j_1,\dots,j_{r-1},i,j_{r+1},\dots,j_d}, 
\text{ if } d<i\le j \text{ or } i\le j < d,\\
0, \text{ otherwise }.
\end{cases}
\end{equation*}
We note that $V^a_{\om_d}$ is no longer isomorphic to $\Lambda^d(V^a_{\om_1})$.
\end{example}

The representations $V_\la^a$ for $\g=\msl_n$ were studied in \cite{FFoL}. 
In particular, the following lemma is proved in \cite{FFoL} (see the proof of Theorem $3.11$):

\begin{lem}\label{emb}
Let $\la,\mu\in P_+$. Then
\[
V_{\la+\mu}^a\simeq \U(\g^a)(v_\la\T v_\mu)\hk V_\la^a\T V_\mu^a
\]
as $\g^a$-modules.
\end{lem}

We now define the degenerate flag varieties $\Fl_\la^a$.
Let $[v_\la]\in\bP(V_\la^a)$ be the line $\bC v_\la$. 
\begin{dfn}
The variety $\Fl_\la^a\hk\bP(V_\la^a)$ is a closure of the
$G^a$-orbit of $[v_\la]$,
\[
\Fl_\la^a=\ov{G^a [v_\la]}=\ov{\bG^M v_\la}\hk \bP(V_\la^a).
\]
\end{dfn}

We note that the orbit $G [v_\la]\hk \bP(V_\la)$ coincides with its closure, but
the orbit $G^a  [v_\la]$ does not. More precisely,
since the Borel subgroup $B$ stabilizes $[v_\la]$,
the orbit
$G^a  [v_\la]=\bG_a^M [v_\la]$ is an affine space $\bC^{M_\la}$,
where
\[
M_\la=\#\{\al\in Q_+:\ (\la,\al)>0\}.
\]
In fact,
\[
\bG_a^M  [v_\la]=\left\{\exp\left(\sum_{\al\in Q_+} c_\al f_\al\right)[v_\la],\ c_\al\in\bC\right\}.
\]
Since $f_\al v_\la=0$ iff $(\la,\al)=0$, we obtain
$\bG^M_a  [v_\la]\simeq \bC^{M_\la}$.

\begin{rem}
The variety $\Fl_\la^a$ is a $\bG_a^{M_\la}$-equivariant compactification of the affine
space  $\bC^{M_\la}$. In other words, $\Fl_\la^a$ are the so-called $\bG_a^M$-varieties
(see \cite{HT}, \cite{A}).
\end{rem}

\section{Pl\" ucker relations, coordinate rings and PBW-tableaux: the degenerate case}
In this section we state main results of this paper. The proofs are given in the 
following sections.

Starting from this section we assume that $\g=\msl_n$. We first consider the varieties 
$\Fl^a_{\om_d}$.

\begin{lem}\label{fund}
$\Fl_{k\om_d}^a\simeq \Fl_{k\om_d}$ for all $d=1,\dots,n-1$, $k\ge 0$.
\end{lem}
\begin{proof}
We note that for a fixed $d$ the elements $f_\al$ with $(\al,\om_d)\ne 0$ pairwise commute in
$\msl_n$. Therefore, 
\[
\Fl_{\om_d}=\ov{\left\{\exp\left(\sum_{\al: (\al,\om_d)\ne 0} c_\al f_\al\right)[v_{\om_d}], 
c_\al\in\bC\right\}}
\simeq \Fl^a_{\om_d}.
\]
\end{proof}

\begin{rem}
The reason that Lemma \ref{fund} holds for all fundamental weights is
that all $\omega_d$ are cominuscule  in type $A$ (the radical, corresponding
to each $\omega_d$ is abelian, see \cite{FFL}, \cite{FL}). This is not true in general 
for other types, so
for general $\g$ the varieties $\Fl_{\om_d}$ and $\Fl_{\om_d}^a$ do not coincide
for all $d$.
\end{rem}

Because of Lemma \ref{emb}, 
we have the following realization of the varieties $\Fl_\la^a$.
\begin{prop}
For $\la\in P_+$ let $1\le d_1<\dots < d_s<n$ be the set of all numbers $d$ such that
$(\la,\al_d)\ne 0$. Then
\[
\Fl^a_\la\simeq \overline{\bG_a^M ([v_{\om_{d_1}}]\times\dots\times [v_{\om_{d_s}}])}\hk
Gr(d_1,n)\times\dots\times Gr(d_s,n).
\]
\end{prop}

\begin{rem}
We note that since all $f_\al$ commute in $V_\la^a$, we have
\begin{multline}
\exp\left(\sum_{\al\in Q_+} c_\al f_\al\right) 
\left([v_{\om_{d_1}}]\times\dots\times [v_{\om_{d_s}}]\right)=\\
\exp\left(\sum_{\al: (\al,\om_{d_1})\ne 0} c_\al f_\al\right) [v_{\om_{d_1}}]\times
\dots\times \exp\left(\sum_{\al: (\al,\om_{d_s})\ne 0} c_\al f_\al\right)[v_{\om_{d_s}}].
\end{multline}
This formula is specific for the degenerate situation and is not true in the classical case.
The reason is that the operators $f_\al$ with $(\al,\la)=0$ act trivially on $V_\la^a$,
but not on $V_\la$ itself. 
\end{rem}

\subsection{Degenerate Pl\" ucker relations}
We introduce the notation
$$\Fl^a(d_1,\dots,d_s)=\Fl^a_{\om_{d_1}+\dots + \om_{d_s}},\ 1\le d_1<\dots < d_s<n.$$
\begin{dfn}
Let $I^a(d_1,\dots,d_s)$ be an ideal in the polynomial ring in variables
$X^a_{j_1,\dots,j_d}$, $d=d_1,\dots,d_s$, $1\le j_1<\dots <j_d<n$, generated by the
elements $R^{k;a}_{L,J}$ given below. 
These elements are labeled by a pair of numbers $p\ge q$, $p,q\in\{d_1,\dots,d_s\}$, 
by an integer $k$, $1\le k\le q$ and by sequences 
$L=(l_1,\dots,l_p)$, $J=(j_1,\dots,j_q)$, which  are arbitrary subsets of the set 
$\{1,\dots,n\}$. The generating elements are given by the 
formulas 
\begin{equation}
R^{k;a}_{L,J}=X^a_{l_1,\dots,l_p} X^a_{j_1,\dots,j_q} - 
\sum_{1\le r_1<\dots <r_k\le p} 
X^a_{l'_1,\dots,l'_p} X^a_{j'_1,\dots,j'_q},
\end{equation}
where the terms of $R^{k;a}_{L,J}$ are the terms of $R^{k}_{L,J}$ \eqref{PR} 
(with a superscript $a$, to be precise) such that 
\begin{equation}\label{new}
\{l_{r_1},\dots,l_{r_k}\}\cap \{q+1,\dots,p\}=\emptyset.
\end{equation}
\end{dfn}

\begin{rem}
The initial term $X^a_{l_1,\dots,l_p} X^a_{j_1,\dots,j_q}$ is also subject to the condition
\eqref{new}, i.e. it is not present in 
$R^{k;a}_{L,J}$ if $\{j_1,\dots,j_k\}\cap \{q+1,\dots,p\}\ne\emptyset$.  
\end{rem}

\begin{rem}\label{mindeg}
Let us write the relation $R^{k}_{L,J}$ as $\sum_i X_{L^{(i)}} X_{J^{(i)}}$.
Then $R^{k;a}_{L,J}$ is equal to the sum of terms $X^a_{L^{(i_0)}} X^a_{J^{(i_0)}}$
such that the sum of PBW degrees \eqref{PBWdegree} 
$\deg L^{(i_0)}+\deg J^{(i_0)}$ is minimal among 
all sums $\deg L^{(i)}+\deg J^{(i)}$. In fact, first note that among the numbers 
$l_1,\dots,l_p$ there exists at least one $k$-tuple $l_{r_1},\dots,l_{r_k}$,
$r_1<\dots r_k$ such that 
\[
\{l_{r_1},\dots,l_{r_k}\}\cap \{q+1,\dots,p\}=\emptyset.
\]
Hence without loss of generality we can assume that 
$$\{j_1,\dots,j_k\}\cap \{q+1,\dots,p\}=\emptyset.$$ Now assume that a pair
$(L',J')$ is obtained from $(L,J)$ by interchanging $(l_{r_1},\dots,l_{r_k})$
with $(j_1,\dots,j_k)$. Then
\[
\deg L'+\deg J'=\deg L + \deg J + \#\{i:\ q+1\le l_{r_i}\le p\}.
\]
Therefore, to obtain $R^{k;a}_{L,J}$ from $R^{k}_{L,J}$, one has to pick the terms of 
minimal PBW-degree.
\end{rem}

Remark \ref{mindeg} implies the following:
\begin{cor}
The ideal $I^a(d_1,\dots,d_s)$ is homogeneous with respect to the PBW-degree.
\end{cor}

We now work out several examples.
\begin{example}
Let $s=1$. Then $I^a(d)=I(d)$, since there are no numbers $l$ such that 
$d+1\le l\le d$ and thus $R^{k;a}_{L,J}=R^{k}_{L,J}$ (up to a superscript $a$ 
in the notations of variables $X_J$).
\end{example}

\begin{example}
Let $\g=\msl_3$. Then the only nontrivial ideal is $I^a(1,2)$. It is generated by a single 
element
\[
R^{1;a}_{(1,2),(3)}=X^a_{1,2}X^a_3 + X^a_{2,3}X^a_1.
\]
We note that 
\[
R^{1}_{(1,2),(3)}=X_{1,2}X_3 -X_{1,3}X_2 + X_{2,3}X_1.
\]
The middle term is missing in $R^{1;a}_{(1,2),(3)}$, because 
\[
\deg (1,2)+\deg(3)=\deg(2,3)+\deg(1)=1, \text{ but } \deg(1,3)+\deg(2)=2,
\] 
or, equivalently, because $2$ satisfies $1+1\le 2\le 2$ (here $q=1,p=2$).
\end{example}

\begin{example}
Let $\g=\msl_4$. Consider the ideal $I^a(1,2)$. This ideal is generated by
the following elements:
\begin{gather*}
R_{(1,2),(3)}^{1;a}=X^a_{1,2}X^a_3 + X^a_{2,3}X^a_1,\ 
R_{(1,2),(4)}^{1;a}=X^a_{1,2}X^a_4 + X^a_{2,4}X^a_1,\\
R_{(1,3),(4)}^{1;a}=X^a_{1,3}X^a_4 - X^a_{1,4}X^a_3 + X^a_{3,4}X^a_1,\
R_{(2,3),(4)}^{1;a}=X^a_{2,3}X^a_4 - X^a_{2,4}X^a_3,\\
R_{(1,2),(3,4)}^{1;a}=X^a_{1,2}X^a_{3,4} - X^a_{1,3}X^a_{2,4} + X^a_{2,3}X^a_{1,4}.
\end{gather*}
We note that the relation $R_{(1,2),(3,4)}^{1;a}$ is exactly the one defining the Grassmanian
$Gr(2,4)$.
\end{example}

\begin{example}\label{I(1,n-1)}
For an arbitrary $n$, consider the ideal $I^a(1,n-1)$. This ideal is generated by a single element
\[
R_{(1,\dots,n-1),(n)}^{1;a}=X^a_{1,\dots,n-1}X^a_n - (-1)^{n} X^a_{2,\dots,n}X^a_1.
\]
In the classical case the corresponding relation is given by
\[
R_{(1,\dots,n-1),(n)}^{1}=\sum_{i=1}^n (-1)^{n-i} X_{1,\dots,\hat{i}\dots,n}X_i.
\]
\end{example}

The following theorem is the main statement of our paper, it will be proved in Section \ref{geom}.
\begin{thm}\label{main}
The variety $\Fl^a(d_1,\dots,d_s)\hk \times_{i=1}^s \bP(\Lambda^{d_i} \bC^n)$ is 
defined by the ideal $I^a(d_1,\dots,d_s)$.
\end{thm}

\begin{example}
Let $s=2$, $d_1=1$, $d_2=n-1$. The variety $\Fl^a(1,n-1)$ is embedded into the product 
$\bP^{n-1}\times \bP^{n-1}=\bP(V^a_{\om_1})\times \bP(V^a_{\om_{n-1}})$. The elements $f_\al$ acting
nontrivially on $V^a_{\om_1}$ are $f_{1,i}$, $i=1,\dots,n-1$, they act as
$f_{1,i}v^a_j=\delta_{1,j}v^a_{i+1}$. Similarly, the elements $f_\al$ 
acting
nontrivially on $V_{\om_{n-1}}$ are $f_{i,n-1}$, $i=1,\dots,n-1$, they act as
\[
f_{i,n-1}v^a_{1,\dots,\hat{j}\dots,n}=(-1)^{n-i+1}\delta_{j,n} v^a_{1,\dots,\hat{i}\dots,n}.
\]
Therefore the coordinates of a point
\[
\exp\left(\sum_{i=1}^{n-1} c_{1,i}f_{1,i}+\sum_{i=1}^{n-1} c_{i,n-1}f_{i,n-1}\right)
\left([v_{\om_1}]\T [v_{\om_{n-1}}]\right)
\]
are given by
\[
X^a_1=1, X^a_i=c_{1,i-1}\ (i>1),\ X^a_{1,\dots,n-1}=1, 
X^a_{1,\dots,\hat{i}\dots,n}=(-1)^{n-i+1} c_{i,n-1}.
\]
Thus, the open orbit of $\Fl^a(1,n-1)$ is the following set of points of the product of two
projective spaces:
\[
\left[v^a_1+\sum_{i=2}^n c_{1,i-1}v^a_i\right]\times \left[v^a_{1,\dots,n-1}+
\sum_{i=1}^{n-1} (-1)^{n-i-1}c_{i,n-1}v^a_{1,\dots\hat{i},\dots,n}\right].
\]
Therefore, the closure of the open cell is defined by a single algebraic relation
\[
X^a_{1,\dots,n-1}X^a_n - (-1)^{n} X^a_{2,\dots,n}X^a_1=0,
\]
which agrees with Example \ref{I(1,n-1)}. Recall that the classical variety
$\Fl(1,n-1)$ is defined by a single relation $R^{1}_{(1,\dots,n-1),(n)}$.
\end{example}

It turns out that our degeneration is flat. 
We will prove the following proposition.
\begin{prop}
The varieties $\Fl=\Fl(d_1,\dots,d_s)$ and $\Fl^a=\Fl^a(d_1,\dots,d_s)$ can be connected 
by a flat family $\Fl^t$ such that $\Fl^0\simeq \Fl^a$ and $\Fl^t\simeq \Fl$, if $t\ne 0$.  
\end{prop}

\subsection{Coordinate rings: quotient polynomial rings}
Define a quotient algebra 
\[
Q^a(d_1,\dots,d_s)=\bC[X^a_{i_1,\dots,i_d}]/I^a(d_1,\dots,d_s),
\]
where the variables $X^a_{i_1,\dots,i_d}$ are labeled by sequences 
$1\le i_1<\dots<i_d\le n$, $d=d_1,\dots,d_s$. 

\begin{thm}
The algebra $Q^a(d_1,\dots,d_s)$ is isomorphic to the algebra
\[
\bigoplus_{\la\in P^+(d_1,\dots,d_s)} (V_{\la}^a)^*,
\] 
where the multiplication $(V_{\la}^a)^*\T (V_{\mu}^a)^*\to (V_{\la+\mu}^a)^*$
is induced by the embedding $V_{\la+\mu}^a\hk V_{\la}^a\T V_{\mu}^a$ 
(see Lemma \ref{emb}).
\end{thm}

\subsection{PBW-semistandard tableaux: polynomial subring}
A crucial statement about the ideal $I^a$ is that it is prime. As in the classical case
this follows from the realization of the quotient algebra as a subalgebra of the polynomial 
algebra.  
In the classical case the key role in the whole picture plays a notion of a
semistandard tableau. We introduce the PBW-version here.

Recall that for a partition $\la=(\la_1\ge\dots\ge \la_{n-1}\ge 0)$ we denote by $Y_\la$ 
the corresponding Young diagram. 
We denote by $\mu_j$ the length of the $j$-th column. 

\begin{dfn}\label{t}
A PBW-tableau of shape $\la$ is a filling $T_{i,j}$ of 
the Young diagram $Y_\la$ with  numbers $1,\dots,n$. The number
$T_{i,j}\in\{1,\dots.n\}$ is attached to the box $(i,j)$. 
The filling  $T_{i,j}$ has to satisfy the following properties:
\begin{itemize}
\item If $T_{i,j}\le \mu_j$, then $T_{i,j}=i$,
\item If $i_1<i_2$ and $T_{i_1,j}\ne i_1$, then
$T_{i_1,j} > T_{i_2,j}$. 
\end{itemize} 
\end{dfn}

\begin{rem}
Usually the entries of a tableau increase from up to down and from left to right.
The reason why it is more natural for us to change the directions is two fold:
first, we want to make difference between numbers, which are smaller than the length
of a column and the others. Second, our convention better fits to the formalism 
of Vinberg's patterns \cite{V}, \cite{FFoL}, see the proof of Proposition \eqref{bij}. 
\end{rem}

\begin{rem}
Let us explain the reason why we want $T_{i,j}\le \mu_j$ to imply 
$T_{i,j}=i$. Assume that $Y_\la$ consists of one column  of length $d$.
Then we have $V_\la^a\simeq V_{\om_d}^a$. As in the classical case we want the shape $\la$
PBW-tableaux to label the standard basis of $V_{\om_d}^a$. So let us
take a basis vector $v^a_{J}=v^a_{j_1,\dots,j_d}$, $j_1<\dots <j_d$. 
This vector is obtained from the highest weight vector $v^a_{1,\dots,d}$ (which is the 
image of $v_1\wedge\dots\wedge v_d $)
by application of several operators $f_\al$.
Note that the 
operators $f_{i,j}$ with $1\le i\le j<d$ act trivially on 
$V_{\om_d}^a$. Therefore, whenever a number $j\le d$ appears in $J$,
it means that no $f_\al$ has been applied to $v_j$.   Hence it is natural to reorder
the vectors $v_j$ in the product $v_{j_1}\wedge\dots\wedge v_{j_d}$ in such a way
that if $j_r\le d$ then $j_r$ appears at the $j_r$-th place.  
\end{rem}

We note that $T_{i,j}\ge i$ for any pair $i,j$.

\begin{dfn}\label{sst}
A semistandard PBW-tableau of shape  $\la$ is a PBW-tableau of shape $\la$  
subject to the condition:
\begin{itemize}
\item
for any $j>1$ and any $i$  there exists $i_1\ge i$ such that $T_{i_1,j-1}\ge T_{i,j}$.
\end{itemize}
\end{dfn}

We denote the set of semistandard PBW-tableaux of shape $\la$ by $ST^a_\la$.

The shape $\la$ semistandard PBW-tableaux parametrize a basis of $V_\la^a$.
In particular
\begin{prop}
The number of semistandard PBW-tableaux is equal to $\dim V_\la$.
\end{prop}

\begin{example}
Let $\la=(\underbrace{1,\dots,1}_d)$. 
Then each sequence $1\le i_1<\dots <i_d\le n$
gives rise to a semistandard PBW-tableau of shape $(d)$.
\end{example}

\begin{example}
Let $\la=(k)$ be a one part partition. Then the shape $\la$ semistandard PBW-tableaux
are of the following form:
\[
\begin{picture}(160,20)
\put(0,0){\line(1,0){160}}
\put(0,10){\line(1,0){160}}
\put(0,0){\line(0,1){10}}
\put(10,0){\line(0,1){10}}
\put(30,0){\line(0,1){10}}
\put(40,0){\line(0,1){10}}
\put(80,0){\line(0,1){10}}
\put(90,0){\line(0,1){10}}
\put(110,0){\line(0,1){10}}
\put(120,0){\line(0,1){10}}
\put(130,0){\line(0,1){10}}
\put(150,0){\line(0,1){10}}
\put(160,0){\line(0,1){10}}
\put(2,2){n}
\put(17,3){...}
\put(32,2){n}
\put(55,3){.\ .\ .}
\put(82,2){2}
\put(97,3){...}
\put(112,2){2}
\put(122,2){1}
\put(137,3){...}
\put(152,2){1}
\end{picture}
\]
Here a number $i=1,\dots,n$ appears $a_i$ times, $a_i\ge 0$ and $\sum_{i=1}^n a_i=k$.  
In particular, the number of semistandard PBW-tableaux is equal to 
$\dim S^k(\bC^n)=\dim V_{k\om_1}$.
\end{example}

\begin{example}
Let $\g=\msl_3$, $\la=(2,1)$. Then the list of the semistandard PBW-tableaux is as follows:\\
\[
\begin{picture}(20,20)
\put(0,0){\line(1,0){10}}
\put(0,10){\line(1,0){20}}
\put(0,20){\line(1,0){20}}
\put(0,0){\line(0,1){20}}
\put(10,0){\line(0,1){20}}
\put(20,10){\line(0,1){10}}
\put(2,11){1}
\put(12,11){1}
\put(2,1){2}
\end{picture}\ , \
\begin{picture}(20,20)
\put(0,0){\line(1,0){10}}
\put(0,10){\line(1,0){20}}
\put(0,20){\line(1,0){20}}
\put(0,0){\line(0,1){20}}
\put(10,0){\line(0,1){20}}
\put(20,10){\line(0,1){10}}
\put(2,11){1}
\put(12,11){2}
\put(2,1){2}
\end{picture}\ ,\
\begin{picture}(20,20)
\put(0,0){\line(1,0){10}}
\put(0,10){\line(1,0){20}}
\put(0,20){\line(1,0){20}}
\put(0,0){\line(0,1){20}}
\put(10,0){\line(0,1){20}}
\put(20,10){\line(0,1){10}}
\put(2,11){1}
\put(12,11){1}
\put(2,1){3}
\end{picture}\ ,\
\begin{picture}(20,20)
\put(0,0){\line(1,0){10}}
\put(0,10){\line(1,0){20}}
\put(0,20){\line(1,0){20}}
\put(0,0){\line(0,1){20}}
\put(10,0){\line(0,1){20}}
\put(20,10){\line(0,1){10}}
\put(2,11){1}
\put(12,11){2}
\put(2,1){3}
\end{picture}\ ,\
\begin{picture}(20,20)
\put(0,0){\line(1,0){10}}
\put(0,10){\line(1,0){20}}
\put(0,20){\line(1,0){20}}
\put(0,0){\line(0,1){20}}
\put(10,0){\line(0,1){20}}
\put(20,10){\line(0,1){10}}
\put(2,11){1}
\put(12,11){3}
\put(2,1){3}
\end{picture}\ ,\
\begin{picture}(20,20)
\put(0,0){\line(1,0){10}}
\put(0,10){\line(1,0){20}}
\put(0,20){\line(1,0){20}}
\put(0,0){\line(0,1){20}}
\put(10,0){\line(0,1){20}}
\put(20,10){\line(0,1){10}}
\put(2,11){3}
\put(12,11){1}
\put(2,1){2}
\end{picture}\ ,\
\begin{picture}(20,20)
\put(0,0){\line(1,0){10}}
\put(0,10){\line(1,0){20}}
\put(0,20){\line(1,0){20}}
\put(0,0){\line(0,1){20}}
\put(10,0){\line(0,1){20}}
\put(20,10){\line(0,1){10}}
\put(2,11){3}
\put(12,11){2}
\put(2,1){2}
\end{picture}\ ,\
\begin{picture}(20,20)
\put(0,0){\line(1,0){10}}
\put(0,10){\line(1,0){20}}
\put(0,20){\line(1,0){20}}
\put(0,0){\line(0,1){20}}
\put(10,0){\line(0,1){20}}
\put(20,10){\line(0,1){10}}
\put(2,11){3}
\put(12,11){3}
\put(2,1){2}
\end{picture}\ .
\]
\end{example}

\begin{example}
Let $\g=\msl_3$ and $\la=(m_1+m_2,m_2)$ be an arbitrary partition. Then the semistandard 
PBW-tableaux are of the form
\[
\begin{picture}(240,30)
\put(0,0){\line(1,0){120}}
\put(0,10){\line(1,0){240}}
\put(0,20){\line(1,0){240}}
\put(0,0){\line(0,1){20}}
\put(10,0){\line(0,1){20}}
\put(30,0){\line(0,1){20}}
\put(40,0){\line(0,1){20}}
\put(50,0){\line(0,1){20}}
\put(70,10){\line(0,1){10}}
\put(80,10){\line(0,1){10}}
\put(90,10){\line(0,1){10}}
\put(110,0){\line(0,1){20}}
\put(120,0){\line(0,1){20}}
\put(130,10){\line(0,1){10}}
\put(150,10){\line(0,1){10}}
\put(160,10){\line(0,1){10}}
\put(170,10){\line(0,1){10}}
\put(190,10){\line(0,1){10}}
\put(200,10){\line(0,1){10}}
\put(210,10){\line(0,1){10}}
\put(230,10){\line(0,1){10}}
\put(240,10){\line(0,1){10}}
\put(2,1){3}
\put(2,11){1}
\put(17,2){...}
\put(17,12){...}
\put(32,1){3}
\put(32,11){1}
\put(42,1){2}
\put(42,11){3}
\put(57,12){...}
\put(72,11){3}
\put(82,11){1}
\put(97,12){...}
\put(112,11){1}
\put(112,1){2}
\put(42,1){2}
\put(57,2){\ .\ .\ .\ .\ .\ .}
\put(122,11){3}
\put(137,12){...}
\put(152,11){3}
\put(162,11){2}
\put(177,12){...}
\put(192,11){2}
\put(202,11){1}
\put(217,12){...}
\put(232,11){1}
\end{picture},
\]
with the restriction $T_{1,m_2+1}\ne 3$ unless $T_{1,m_2}=3$. 
More precisely, 
\begin{multline*}
T_{1,1}=\dots =T_{1,k_1}=T_{1,k_1+k_2+1}=\dots=T_{1,m_2}=\\
T_{1,m_2+k_3+k_4+1}=\dots =T_{1,m_1+m_2}=1,
\end{multline*}
\[
T_{1,m_2+k_3+1}=\dots =T_{1,m_2+k_3+k_4}=T_{2,k_2+1}=\dots=T_{2,m_2}=2,\]
\begin{multline*}
T_{2,1}=\dots =T_{2,k_1}=T_{1,k_1+1}=\dots=T_{1,k_1+k_2}=\\ 
T_{1,m_2+1}=\dots =T_{m_2+1+k_3}=3,
\end{multline*}
$k_1, k_2, k_3,k_4\ge 0$ and $k_3=0$ unless $k_1+k_2=m_2$.
One can easily 
show that the number of such tableaux is equal to $(m_1+1)(m_2+1)(m_1+m_2+1)/2$,
which coincides with the dimension of $V_{m_1\om_1+m_2\om_2}$.
\end{example}

Given a PBW-tableau $T$, we define an element $D^a_T$ in the ring
$\bC[Z_{i,j}]$, $i<j$. First, assume that we have  a sequence of 
numbers $1\le i_1<\dots < i_d\le n$. Let $m\le d$ 
be a number such that $i_m\le d$ and $i_{m+1}>d$.
Then $D^a_{i_1,\dots,i_d}$ is the determinant of the matrix $D^a_{\al,\beta}$, $\al,\be=1,\dots,d$ 
\begin{equation}\label{Da}
D^a_{\al,\be}=
\begin{cases}
Z_{\al,i_\beta},\text{ if } \al\ne i_1,\dots,i_m, \beta>m,\\
0, \text{ if } \al\ne i_1,\dots,i_m, \beta\le m,\\
1,\ \text{ if }\al=i_1,\dots,i_m, \be=\al,\\
0,\ \text{ if } \al=i_1,\dots,i_m, \be\ne\al.
\end{cases}
\end{equation}
Now assume we are given a length $d$ column of pairwise distinct numbers.
Then there exists
a permutation $\sigma\in S_d$ such that $\sigma T$ has increasing entries.
We then define $D_T^a=(-1)^\sigma D^a_{\sigma T}$.
If $\la$ has more than one column, the element $D^a_T$ is the product of  
of the polynomials, corresponding to its columns.

\begin{prop}
For any relation $R^{k;a}_{L,J}$ one has $R^{k;a}_{L,J}(D^a_{i_1,\dots,i_d})=0$.
\end{prop}

We will prove the following proposition:
\begin{prop}\label{D}
Fix a partition $\la$.
% Assume that all column of $\la$ are of the length
%$d_1,\dots,d_s$. Then 
The polynomials $D^a_T$,  $T\in ST^a_\la$, are linearly independent
\end{prop}

For a tableau $T$ of shape $\la$ we define a n element $X^a_T\in Q^a(d_1,\dots,d_s)$ by the formula
\[
X^a_T=\prod_{j=1}^{\la_1} X^a_{T_{1,j}}\dots X^a_{T_{\mu_j,j}}.
\]
Proposition \ref{D} allows to prove the following theorem. 

\begin{thm}
The elements $X^a_T$, $T\in ST^a_\la$, 
form a basis of $Q^a(d_1,\dots,d_s)$. 
\end{thm}

Recall the notation
\[
P^+(d_1,\dots,d_s)=\{m_1\om_{d_1}+\dots +m_s\om_{d_s},\ m_i\in \bZ_{\ge 0}\}. 
\]

\begin{cor}
$Q^a(d_1,\dots,d_s)\simeq \bigoplus_{\la\in P^+(d_1,\dots,d_s)} (V_\la^a)^*$. 
\end{cor}

\section{Coordinate rings and PBW-tableaux}
In this section we start proving the results from the previous section.
Throughout the section we fix a sequence of integers $1\le d_1 <\dots < d_s <n$
and omit this sequence in the notations writing
$Q^a$, $I^a$ etc instead of $Q^a(d_1,\dots,d_s)$, $I^a(d_1,\dots,d_s)$ etc. 

Consider the spaces
\[
\bar Q^a=\bigoplus_{\la\in P^+(d_1,\dots,d_s)} (V_\la^a)^*.
\]
This space can be endowed with the structure of an algebra with the multiplication
$(V_\la^a)^*\T (V_\mu^a)^*\to (V_{\la+\mu}^a)^*$ coming from the embedding
$V_{\la+\mu}^a\hk V_\la^a\T V_\mu^a$. 
We note that $\bar Q^a$ is PBW-graded, since the embeddings $V^a_{\la+\mu}\hk V^a_\la\T V^a_\mu$
is compatible with the PBW-grading.
The algebra
$\bar Q^a$ is generated by the subspaces $(V_{\om_{d_i}}^a)^*$. Recall
the elements $X^a_{i_1,\dots,i_d}\in (V_{\om_d}^a)^*$.

\begin{lem}
The elements $X^a_{i_1,\dots,i_{d_l}}$, $l=1,\dots,s$ satisfy the relations $R^{k;a}_{L,J}$ in
$\bar Q^a$.
\end{lem}
\begin{proof}
The elements $X_{i_1,\dots,i_{d_l}}$, $l=1,\dots,s$ do satisfy relations $R^{k}_{L,J}$ 
in the classical algebra $Q=\bigoplus_{\la\in P^+} V_\la^*$ and the relations
$R^{k;a}_{L,J}$ are simply the lowest degree terms with respect to the PBW grading
(see Remark \ref{mindeg}). 
Since the algebra $\bar Q^a$ is PBW graded, our lemma follows. 
\end{proof}

Recall the algebra $Q^a=\bC[X^a_{i_1,\dots,i_d}]/I^a$, $d=d_1,\dots,d_s$, 
$1\le i_1<\dots <i_d\le n$.
\begin{cor}\label{bar}
There is a surjection of algebras $Q^a\to\bar Q^a$.
\end{cor}

We will prove that this surjection is an isomorphism. 
Consider the decomposition
\[
Q^a=\bigoplus_{\la\in P^+(d_1,\dots,d_s)} Q^a_\la,
\]
where for $\la=m_1\om_{d_1} + \dots + m_s\om_{d_s}$ the subspace 
$Q^a_\la$ is spanned by the monomials $X^a_{J^{(1)}}\dots X^a_{J^{(r)}}$, where
\[
\#\{i:\ J^{(i)} \text{ has } d_l \text{ entries}\}=m_l, \ l=1,\dots,s.
\]
Our goal is to show that $\dim Q^a_\la=\dim V_\la$.

Recall  definitions \ref{t}, \ref{sst} of the PBW-tableau and of the 
semistandard PBW-tableau.

\begin{lem}\label{>}
Let $T$ be a semistandard PBW-tableau. Then:
\begin{itemize}\label{prop}
\item If $j_1<j_2$ and $T_{i,j_1}=T_{i,j_2}\ne i$, then $T_{i,j}=T_{i,j_1}$ for all 
$j_1\le j\le j_2$.
\item If $i_1\le i_2$, $j_1\le j_2$ and $T_{i_1,j_1}\ne i_1$,
$T_{i_2,j_2}\ne i_2$, then $T_{i_1,j_1}\ge T_{i_2,j_2}$. If in addition
$i_1<i_2$, then $T_{i_1,j_1}> T_{i_2,j_2}$.  
\end{itemize}
\end{lem}
\begin{proof}
Let us prove the first statement. Assume that there exists $j$, $j_1<j<j_2$ such that
$T_{i,j}\ne T_{i,j_1}$. Then $T_{i,j}=i$, because otherwise from the definition of 
a PBW-tableau we have $T_{i,j_1}\ge T_{i,j}\ge T_{i,j_2}$ and one of the inequalities is
strict, which is a contradiction.
So $T_{i,j}<T_{i,j_2}$. By definition of a semistandard PBW-tableau, there
exists a sequence $i_1\le i_2\le\dots\le i_{j_2-j_1}$ such that
$i_1\ge i$ and 
\[
T_{i,j_2}\le T_{i_1,j_2-1}\le T_{i_2,j_2-2}\le\dots\le T_{i_{j_2-j_1},j_1}.
\]  
Since $T_{i,j}<T_{i,j_2}$, we have $i_{j_2-j}>i$ and hence $i_{j_2-j_1}>i$.
This gives $$T_{i,j_2}\le T_{i_{j_2-j_1},j_1}<T_{i,j_1},$$ which gives a contradiction.

The second statement of the lemma can be proved similarly. 
\end{proof}

\begin{cor}\label{tab}
Let $T$ be a semistandard PBW-tableau. For $l>k$ let $A_{col}(k,l)$ be the minimal number
such that $T_{k,A_{col}(k,l)+1}=l$. Then
\begin{itemize}
\item there exists $i\ge k$ such that $T_{i,A_{col}(k,l)}\ge l$;
\item for $j>A_{col}(k,l)$ and $i>k$ $T_{i,j}<l$.
\end{itemize}
\end{cor}
%\begin{proof}
%The statement of the lemma can be rephrased as follows: let $A_{col}(k,l)$
%be the number such that $l$ appears in the $k$-th row first time (leftmost) 
%at $(A_{col}(k,l)+1)$-st place. Then $A_{col}(k,l)$ is the maximal number
%of a column which contains in rows $k,k+1,\dots,n-1$ an entry greater than or equal
%to $l$. This is proved in two steps: first, we show that there exists $i\ge k$ such that
%$T_{i,A_{col}(k,l)}\ge l$. Second, we show that if $j>A_{col}(k,l)$, then for all $i\ge k$
%$T_{i,A_{col}(k,l)}< l$. First step follows from $T_{k,A_{col}(k,l)+1} =l$ and
%definition \ref{sst}. Second statement follows from Lemma \ref{prop}.
%\end{proof}

\begin{prop}\label{bij}
The number of shape $\la$ semistandard PBW-tableaux is equal to the 
dimension of $V_\la$.
\end{prop}
\begin{proof}
Let $\la=\sum_{i=1}^{n-1} m_i\om_i$. 
Recall (see \cite{FFoL}, \cite{V}) that the dimension of $V_\la$ is equal to the 
number of Vinberg's configurations.
A configuration is a set of numbers $(s_{i,j})_{1\le i\le j\le n-1}\in\bZ_{\ge 0}$ 
or, equivalently, a set of numbers $s_\al$ labeled by positive roots $\al$ 
(we identify a root $\al=\al_i+\dots +\al_j$, $i\le j$, with the pair $i,j$). 
A configuration is said to be Vinberg's configuration, if
the following condition holds: for any Dyck path
$\bp=(p_0,\dots,p_k)$ with $p(0)=\al_i$, $p(k)=\al_j$ we have 
\begin{equation}\label{V}
\sum_{l=0}^k s_{p_l}\le m_i + \dots + m_j.
\end{equation}
Recall that a Dyck path $\bp=(p_0,\dots,p_k)$ is a sequence of roots of $\msl_n$,
$$p_l=\al_{i_l,j_l}=\al_{i_l}+\dots + \al_{j_l},\ i_l\le j_l,$$
such that either ($i_{l+1}=i_{l}$, $j_{l+1}=j_{l}+1$) or ($i_{l+1}=i_{l}+1$, $j_{l+1}=j_l$).
Let $S_\la$ be the set of Vinberg's configurations. We construct a one-to-one map $\psi$ 
from $S_\la$ to the set $ST^a_\la$ of semistandard PBW-tableaux of shape $\la$.  

Let $\bs\in S_\la$ be some configuration. We define a tableau $\psi (\bs)$
in the following way. We fill the diagram $Y_\la$ from bottom to top.
Let us start with the $(n-1)$-st row (if there is no such row, the procedure described
below is empty). We define
\[
T_{n-1,1}=\dots=T_{n-1,s_{n-1,n-1}}=n
\]
and fill the rest of the $(n-1)$-st row with the numbers $n-1$.

Now assume that the rows from $k+1$ to $n$ are already filled. 
We fill the $k$-th row in the following way. 
Let $A_{col}(k,n)$ be the maximal number such that there exists $k\le i\le n-1$ 
such that $T_{i,A_{col}(k,n)}=n$ (since nothing is present in the $k$-th row at the moment,
we can say $k<i$). We then define
\[
T_{k,A_{col}(k,n)+1}=\dots = T_{k,A_{col}(k,n)+s_{k,n-1}}=n.
\]  
In general, we fill the $k$-th row with the numbers $n,n-1,\dots,k+1$ in the 
decreasing order in the following way. 
Assume that all numbers $n,\dots, l+1$ we want to put in the $k$-th row
are already there. Let us explain how do we put the numbers $l$. Let
$A_{col}(k,l)$ be the maximal number such that there exists
$i=k,\dots,n$ such that $T_{i,A_{col}(k,l)}\ge l$.
We then define
\[
T_{k,A_{col}(k,l)+1}=\dots =T_{k,A_{col}(k,l)+s_{k,l-1}}=l.
\]  
After all iterations of the procedure above (corresponding to the numbers 
$l=n,\dots, k+1$) we fill the rest boxes of the $k$-th row with numbers $k$. 

Note that the procedure above fills the $k$-th row of a tableau with numbers
greater than or equal to $k$.
We have to show that the construction above is well-defined and produces a semistandard
PBW-tableau. For this we have to show that
the numbers we fill the $k$-th row with fit into it. Since the length of the $k$-th row
is equal to $m_k+\dots + m_{n-1}$, it suffices  to show that for all $l>k$  
\begin{equation}\label{Acol}
A_{col}(k,l)+s_{k,l-1}\le m_k+\dots + m_{n-1}.
\end{equation}
We define
\[
A_{row}(k,l)=\min\{i:\ T_{i,A_{col}(k,l)}\ge l\}.
\]
We set $A(k,l)=(A_{row}(k,l),A_{col}(k,l))$.

Assume first that $A_{row}(k,l)\ge l$. Then
\[
A_{col}(k,l)\le m_l+\dots + m_{n-1}
\] 
and since $s_{k,l-1}\le m_k+\dots +m_{l-1}$, the inequality \eqref{Acol} holds.
Now let $A_{row}(k,l)< l$. Then since $T_{A(k,l)}\ge l$, we have 
$T_{A(k,l)}\ne A_{row}(k,l)$. Set
\[
k_1=A_{col}(k,l),\ \ l_1=T_{A(k,l)}.
\]
Note that $k_1\ge k$ and $l_1\ge l$.
Then $A_{row}(k,l)=A_{row}(k_1,l_1) + s_{k_1,l_1-1}$.
Thus \eqref{Acol} is equivalent to
\begin{equation}\label{Acol1}
A_{col}(k_1,l_1) +s_{k_1,l_1-1} + s_{k,l-1}\le m_k+\dots + m_{n-1}.
\end{equation}
As on the previous step, assume first that $A_{row}(k_1,l_1)\ge l_1$. Then 
\[
A_{col}(k_1,l_1)\le m_{l_1}+\dots + m_{n-1}
\] 
and since $s_{k,l-1} + s_{k_1,l_1-1}\le m_k+\dots +m_{l_1-1}$ (recall
that $k_1\ge k$ and $l_1\ge l$ and therefore the points $(k,l-1)$ and $(k_1,l_1-1)$
can be connected by a Dyck path), the inequality \eqref{Acol1} holds.
So let $A_{row}(k_1,l_1)< l_1$. Define
\[
k_2=A_{col}(k_1,l_1),\ \ l_2=T_{A(k_1,l_1)}.
\]
We note that $k_2\ge k_1$ and $l_2\ge l_1$.
The \eqref{Acol1} is equivalent to 
\begin{equation}\label{Acol2}
A_{col}(k_2,l_2) + s_{k_2,l_2-1}+s_{k_1,l_1-1} + s_{k,l-1}\le m_k+\dots + m_{n-1}.
\end{equation}
Note that the points $(k,l)$, $(k_1,l_1)$ and $(k_2,l_2)$, can be connected by a Dyck path.

We continue in the same way until we get $A_{col}(k_p,l_p)=0$. The corresponding inequality 
(analogous to \eqref{Acol}, \eqref{Acol1} and \eqref{Acol2}) follows from the inequality
on the sum of $s_\bullet$ on a single Dyck path.

Now we need to construct an opposite map $\psi^{-1}$ from $ST^a_\la$
to $S_\la$. Given a tableau, we define
\begin{equation}\label{bs}
s_{k,l}=\#\{j:\ T_{k,j}=l+1\}.
\end{equation}
Taking into account Corollary \ref{tab} we prove in the same way as above that 
$\bs$ defined by \eqref{bs} is an element of $S_\la$. Finally, we note that
the two maps constructed above are inverse to each other. 
\end{proof}

In order to establish the isomorphism $Q^a\simeq \bar Q^a$ and to prove that
$I^a$ is prime, we introduce one more algebra $\tilde Q^a$ (actually, we will show
that $\tilde Q^a\simeq Q^a$). 
Let $\tilde Q^a$ be an algebra inside the polynomial ring 
$\bC[Z_{i,j}]_{1\le i<j\le n}$ generated by the elements 
$D_{i_1,\dots,i_d}^a$ (see \eqref{Da}).

\begin{lem}
For all $k$, $L$ and $J$ one has $R^{k;a}_{L,J}(D_{i_1,\dots,i_d}^a)=0$
in $\bC[Z_{i,j}]_{1\le i<j\le n}$.
\end{lem} 
\begin{proof}
We note that the definition of $D^a_{i_1,\dots,i_d}$ can be given in the form of the equality
in $V^a_{\om_d}$:
\[
\exp\left(\sum_{i<j} Z_{i,j}f_{i,j-1}\right)[v^a_{1,\dots,d}]=\sum_{i_1<\dots<i_d} 
D^a_{i_1,\dots,i_d}v^a_{i_1,\dots,i_d}.
\]
The statement that these coefficients satisfy relations from $I^a$ 
is proved in the following section, see Lemma \ref{cell}. 
\end{proof}

\begin{cor}
There exists a surjection of algebras $Q^a\to \tilde Q^a$.
\end{cor}

\begin{lem}\label{linind}
Fix a partition $\la$. Then the elements $D_T^a$, $T\in ST^a_\la$, are linearly independent.   
\end{lem}
\begin{proof}
Let $T$ be a semistandard PBW-tableau. Note that $D^a_T$ contains a monomial
\[
M^a_T=\prod_{i,j:\ T_{i,j}\ne i} Z_{i,T_{i,j}}.
\]
It is convenient to rewrite this monomial in terms of the corresponding 
Vinberg's configuration $\bs=\psi^{-1}(T)$ (see the bijection in the proof of Proposition \ref{bij}). 
Namely it is easy to see that
\[
M^a_\bs=\prod_{i\le l} Z_{i,l+1}^{s_{i,l}}.
\] 
Let us consider the following order on the set $S_\la$. We say 
$\bs > {\bf t}$ if there exists a pair $i_0,l_0$ such that
$s_{i_0,l_0}>t_{i_0,l_0}$ and
\[
s_{i,l}=t_{i,l} \text{ if } (i<i_0 \text{ or } i=i_0, l<l_0).
\]
Note that if $\bs>{\bf t}$, then the monomial $M_\bs^a$ does not appear in
$D^a_{\psi({\bf t})}$. Therefore the elements $D^a_T$ are linearly independent.
\end{proof}

Recall that for a shape $\la$ tableau $T$ we have the elements $X^a_T$ defined by
\[
X^a_T=\prod_{j=1}^{\la_1} X^a_{T_{1,j},\dots,T_{\mu_j,j}}.
\]
\begin{lem}\label{XT}
The elements $X^a_T$, $T\in ST^a_\la$,  span $Q^a$ and hence $\tilde Q^a$ and $\bar Q^a$.
\end{lem}
\begin{proof}
We introduce an order on the set of tableaux of shape $\la$. We say
$T^{(1)}>T^{(2)}$ if there exist $i_0,j_0$ such that  
$T^{(1)}_{i_0,j_0}>T^{(2)}_{i_0,j_0}$ and 
\[
T^{(1)}_{i,j}=T^{(2)}_{i,j} \text{ if } j>j_0 \text{ or } (j=j_0, i>i_0).
\]
Assume that we are given a PBW-tableau $T$, which is not semistandard.
Recall that the elements $D^a_{i_1,\dots,i_d}$ do satisfy relations from
$I^a$. Using these relations we rewrite $D^a_T$ in terms of smaller tableaux.

%So assume that the conditions from .. are violated. First, assume that there exist
%$k, j_1<j_2$ such that $T_{k,j_1}<T_{k,j_2}$, $T_{k,j_1}\ne k$.
%Let $p=\mu_{j_1}$, $q=\mu_{j_2}$. 
%We will use the relation
%\begin{equation}\label{R1}
%R^{k;a}_{(T_{1,j_1},\dots,T_{p,j_1}),(T_{1,j_2},\dots,T_{q,j_2})}.
%\end{equation}
%Since 
%$T_{k,j_1}<T_{k,j_2}$ and  $T_{k,j_1}>p$, we have $T_{k,j_2}>p$. 
%Therefore none of the elements $T_{1,j_2}, \dots, T_{k,j_2}$ belongs to
%the set $\{q+1,\dots,p\}$ and hence the term
%\[
%D^a_{T_{1,j_1},\dots,T_{k,j_1}}D^a_{T_{1,j_2},\dots,T_{k,j_2}}
%\]
%is present in the relation \eqref{R1}. Let us look at other terms in \eqref{R1}.
%Assume we pick arbitrary numbers $1\le i_1<\dots<i_k\le p$ and
%interchange
%\[
%T_{i_1,j_1},\dots, T_{i_k,j_1} \text{ with } T_{1,j_2},\dots,T_{k,j_2}.
%\]
%Since $T_{i_k,j_1}<T_{k,j_2}$, we have (after interchanging) at least $q-k+1$
%elements $T_{i_k,j_1},T_{k+1,j_2},\dots,T_{q,j_2}$ which are smaller than
%$T_{k,j_2}$. Therefore the lowest element in the column $j_2$, which differs from
%the one in $T$ is smaller. Therefore, after interchanging the resulting tableau becomes smaller.

Assume that the condition from definition \ref{sst} is violated. Then there exist
$k,j$ such that $T_{k,j+1}>k$ and for all $i\ge k$ $T_{k,j+1}>T_{i,j}$.
Let $p=\mu_j$, $q=\mu_{j+1}$. Since $T_{k,j+1}>T_{p,j}\ge p$, 
none of the elements $T_{1,j+1}, \dots, T_{k,j+1}$ belong to
the set $\{q+1,\dots,p\}$ and hence the term
\[
X^a_{T_{1,j},\dots,T_{k,j}}X^a_{T_{1,j+1},\dots,T_{k,j+1}}
\]
is present in the relation
\begin{equation}\label{R2}
R^{k;a}_{(T_{1,j},\dots,T_{p,j}),(T_{1,j+1},\dots,T_{q,j+1})}.
\end{equation}
Let us look at other terms in \eqref{R2}.
Assume we pick arbitrary numbers $1\le i_1<\dots<i_k\le p$ and
interchange
\[
T_{i_1,j},\dots, T_{i_k,j}\ \text{ with }\ T_{1,j+1},\dots,T_{k,j+1}.
\]
Since $T_{i,j}<T_{k,j+1}$ for all $i\ge k$, we have (after interchanging) at least 
$q-k+1$
elements $T_{i_k,j},T_{k+1,j+1},\dots,T_{q,j+1}$ which are smaller than
$T_{k,j+1}$. Therefore the lowest element in the column $j_2$, which differs from
the one in $T$ is smaller. Therefore, after interchanging the resulting tableau becomes smaller.
\end{proof}

\begin{thm}
We have isomorphisms $Q^a\simeq \bar Q^a\simeq \tilde Q^a$.
The elements $X_T^a$ labeled by semistandard PBW-tableaux $T$ form a basis of $Q^a$. 
\end{thm}
\begin{proof}
In fact, Lemma \ref{XT} and Proposition \ref{bij} give the inequalities 
$\dim Q_\la^a\le \dim V_\la$ and 
$\dim \tilde Q_\la^a\le \dim V_\la$. From Lemma \ref{linind} we obtain 
$\dim \tilde Q_\la^a= \dim V_\la$. Since $\dim \bar Q_\la^a= \dim V_\la$,
Corollary \ref{bar} finishes the proof.
\end{proof}

\begin{cor}\label{prime}
The ideal $I^a$ is prime.
\end{cor}
\begin{proof}
Follows from the isomorphism $Q^a\simeq \tilde Q^a$, since $\tilde Q^a$ is a subalgebra
in the polynomial algebra.
\end{proof}

We close this section with the following proposition.
\begin{prop}\label{nb}
The elements $X_T$, $T\in ST^a_\la$, form a basis of $Q$.
\end{prop}
\begin{proof}
Since $\# ST^a_\la=\dim V_\la$, it suffices to prove that 
the elements $X_T$, $T\in ST^a_\la$, span $Q$. Recall that in the proof of Lemma \ref{XT}
we have introduced an order on the set of tableau and expressed each $X^a_T$, 
$T\notin ST^a_\la$
as a linear combination of $X_{T'}$ with $T'$ smaller than $T$. For this we have used
the relations $R^{k;a}_{L,J}$. We note that
\[
R^{k}_{L,J}=\tilde R^{k;a}_{L,J} + \sum_i X_{L^{(i)}}X_{J^{(i)}}, 
\] 
where $\deg J^{(i)} + \deg L^{(i)} > \deg J + \deg L$ and $\tilde R^{k;a}_{L,J}$
is obtained from $R^{k;a}_{L,J}$ simply by omitting superscripts $a$ in variables $X^a_J$.
Therefore in $Q$ any element $X_T$, $T\notin ST^a_\la$ can be expressed as a linear combination
of $X_{T'}$, where $T'$ is smaller than $T$ or the sum of PBW-degrees of 
columns of $T'$ is bigger than that of $T$. Since the sum of PBW-degrees of fixed shaped tableaux
is bounded from above, our proposition is proved. 
\end{proof}

\section{The varieties $\Fl^a(d_1,\dots,d_s)$}\label{geom}
\subsection{The ideal $I^a$}
Throughout the section we fix a sequence $1\le d_1<\dots <d_s\le n-1$ and omit 
the numbers $d_i$ in the notations $\Fl^a$, $I^a$, etc.

Let $U\hk \Fl^a$ be the dense $\bG_a^M$-orbit 
\begin{multline*}
U=\left\{\exp\left(\sum_{1\le i\le j\le n-1} c_{i,j}f_{i,j}\right)[v_{\om_{d_1}}]\times\right. 
\dots
\times\\ \left.\exp\left(\sum_{1\le i\le j\le n-1} c_{i,j}f_{i,j}\right)[v_{\om_{d_s}}],\ c_{i,j}\in\bC
\right\}.
\end{multline*}
By definition, the variety $\Fl^a$ is the closure of $U$ in the product of the
projective spaces $\bP(V^a_{\om_{d_i}})$, $i=1,\dots,s$.
Recall the coordinates $X^a_{i_1,\dots,i_d}$ in $V^a_{\om_d}$. 
These coordinates produce
a multi-homogeneous coordinates on the product $\times_{i=1}^s \bP(V^a_{\om_{d_i}})$. 
Let us denote by $X^a_{i_1,\dots,i_d}(c_{i,j})=X^a_{i_1,\dots,i_d}(\bc)$ 
the coordinates of a point
\[
\exp\left(\sum_{1\le i\le j\le n-1} c_{i,j}f_{i,j}\right)
(v^a_{\om_{d_1}}\times \dots \times v^a_{\om_{d_s}})\in \times_{i=1}^s V^a_{\om_{d_i}}.
\]
For instance, $X_{1,\dots,d}=1$.
Let us compute these coordinates explicitly.  
Let $1\le r\le d$ be a number such that $i_r\le d$ and $i_{r+1}>d$. 
Define the numbers $1\le j_1<\dots <j_{d-r}\le d$ by the formula
\[
\{j_1,\dots,j_{d-r}\}=\{1,\dots,d\}\setminus \{i_1,\dots,i_r\}.
\]

\begin{lem}\label{xcij}
We have
\[
X^a_{i_1,\dots,i_d}(\bc)=(-1)^{\sum_{l=1}^r (i_l - l)}
\sum_{\sigma\in S_{d-r}} (-1)^\sigma c_{j_1, i_{r+\sigma(1)}-1}\dots 
c_{j_{d-r}, i_{r+\sigma(d-r)}-1}.
\]
\end{lem}
\begin{proof}
Lemma follows from formula \eqref{action}.
\end{proof}

\begin{cor}
We have $X^a_{i_1,\dots,i_d}(\bc)=\det C$, where $C_{\al,\beta}$ is an $n\times n$ matrix 
defined by
\begin{equation}\label{C}
C_{\al,\be}=
\begin{cases}
c_{\al,i_\beta-1},\text{ if } \al\ne i_1,\dots,i_m, \beta>m,\\
0, \text{ if } \al\ne i_1,\dots,i_m, \beta\le m,\\
1,\ \text{ if }\al=i_1,\dots,i_m, \be=\al,\\
0,\ \text{ if } \al=i_1,\dots,i_m, \be\ne\al.
\end{cases}
\end{equation}
\end{cor}
We note that matrices \eqref{C} and \eqref{Da} are identical up to the identification
of parameters.

\begin{example} 
Note that $X_{1,\dots,d}(\bc)=1$. 
Let $i\le d<j$. Then
\begin{equation}\label{Xc}
X_{1,\dots ,\hat{i}\dots, d,j}(\bc)=(-1)^{d-i} c_{i,j-1}.
\end{equation}
\end{example}

\begin{lem}\label{cell}
Fix a tuple $\bc$. Then the coordinates
$X^a_{i_1,\dots,i_{d_r}}(\bc)$, $r=1,\dots,s$ satisfy the relations from the ideal
$I^a$.  
\end{lem}
\begin{proof}
Let $X_{i_1,\dots,i_d}(\bc)$ be the classical analogues of
$X^a_{i_1,\dots,i_{d}}(\bc)$, i.e.
\[
\exp\left(\sum_{1\le i\le j<n} c_{i,j}f_{i,j}\right) v_{\om_d}=
\sum_{1\le i_1<\dots <i_d\le n} X_{i_1,\dots,i_{d}}(\bc) v_{i_1,\dots,i_d} \in V_{\om_d}.
\] 
For any tuple $\bc$ the coordinates $X_{i_1,\dots,i_d}(\bc)$, $d=d_1,\dots,d_s$
satisfy relations from $I$.
We note that the polynomials $X^a_{i_1,\dots,i_{d_r}}(\bc)$ can be obtained from
$X_{i_1,\dots,i_{d_r}}(\bc)$ simply by taking
the lowest degree (in all variables $c_{i,j}$) terms. This lowest
degree coincides with the PBW-degree of the tuple $i_1,\dots,i_d$.
Since the terms in $R^{k;a}_{L,J}$ are exactly the lowest
degree terms among those in  $R^{k;a}_{L,J}$, the coordinates
$X^a_{i_1,\dots,i_{d_r}}(\bc)$ do satisfy the relations $R^{k;a}_{L,J}$.
\end{proof}

Now consider the set of all zeros $X(I^a)\hk \times_{r=1}^s \bP(V^a_{\om_{d_r}})$
of the ideal $I^a$.
\begin{lem}
Let $W\subset \times_{r=1}^s \bP(V_{\om_{d_r}})$ be an affine space defined by
$X_{1,\dots,d_r}\ne 0$ for all $r=1,\dots,s$. Then
$X(I^a)\cap W=U.$
\end{lem}
\begin{proof}
Let $P_{i_1,\dots,i_{d_r}}$, $r=1,\dots,s$ be the coordinates of a point $P\in X(I^a)\cap W$.   
We note that  for all $l\le d_r < d_m <j$ the relation
$R^{1;a}_{(1,\dots, d_m),(j,1,\dots, \hat{l}\dots, d_r)}$ give
\begin{equation}\label{c=c}
P^a_{1,\dots, d_m}P^a_{j,1,\dots, \hat{l}\dots, d_r} = 
P^a_{j,1,\dots, \hat{l}\dots, d_m}P^a_{1,\dots, d_r}. 
\end{equation}
Assuming $P^a_{1,\dots, d}=1$, relation \eqref{c=c} 
reads as
\begin{equation}\label{deg1}
P^a_{1,\dots, \hat{l}\dots, d_m,j} = 
(-1)^{d_r+d_m}P^a_{1,\dots, \hat{l}\dots, d_r,j} 
\end{equation}
(which agrees with formula \eqref{Xc}). 
We set 
\[
c_{i,j-1}=(-1)^{d_m-i} P_{1,\dots,\hat{i}\dots,d_m,j},
\]
if there exists $m$ such that $i\le d_m<j$ and set $c_{i,j}=0$ otherwise.
Equation \eqref{deg1} guarantees that such tuple $c_{i,j}$ exists. 
Our goal is to show that
\begin{equation}\label{coin}
\exp\left(\sum_{1\le i\le j<n} c_{i,j}f_{i,j}\right) 
\left([v^a_{\om_{d_1}}]\times\dots\times [v^a_{\om_{d_s}}]\right)= P.
\end{equation}
The right and left hand sides belong to the affine space with the coordinates
$X^a_{i_1,\dots,i_r}$ (we assume $X^a_{1,\dots,d_r}=1$). We have defined
$c_{i,j}$ in such a way that the coordinates $X^a_{1,\dots\hat{i},\dots,d_r,j}$
of the left and right hand sides of \eqref{coin} coincide. This implies \eqref{coin},
since there is at most one point in $X(I^a)\cap W$ with the prescribed values of
$X^a_{1,\dots\hat{i},\dots,d_r,j}$. In fact, let $J=(j_1,\dots,j_d)$ be a sequence such that
$\deg J>1$. Let us order $J$ in such a way that $j_1>d$. Then the relation 
$R^{1;a}_{(1,\dots,d),J}$ allows to rewrite $X^a_J$ ($X^a_JX_{1,\dots,d}$ to be precise)
as a polynomial in $X_L^a$ with $\deg L<\deg J$. Now it suffices to note that
the sequences $(1,\dots\hat{i},\dots,d,j)$ are exactly the degree one sequences.   
\end{proof}

For a subset $R\subset \times_{r=1}^s \bP(V^a_{\om_{d_r}})$ let $I(R)$ be the ideal
of multi-homogeneous polynomials vanishing on $R$.
We now prove the main theorem of this section.
\begin{thm}
$I(U)=I^a$. Equivalently, $I(\Fl^a)=I^a$.
\end{thm}
\begin{proof}
We need to prove that if a multi-homogeneous polynomial $F$ in variables 
$X^a_{i_1\dots i_d}$, $d=d_1,\dots,d_s$ vanishes on $U$, then $F\in I^a$. 
For a $d$-tuple $j_1,\dots,j_d$ consider the relation
$R^{1;a}_{(1,\dots,d);(j_1,\dots,j_d)}\in I^a$. This relation is equal to
\[
X^a_{1,\dots,d} X^a_{j_1,\dots,j_d} - \sum X^a_{L} X^a_{J'},
\]
where for all terms $X^a_{L} X^a_{J'}$ we have
\[
\deg X^a_L=1,\  \deg X^a_{J'}=  \deg (j_1,\dots,j_d) - 1.
\]
(Recall the PBW degree $\deg J=\#\{l:\ j_l>d\}$).
Therefore, 
the product $(X^a_{1,\dots,d})^{\deg J} X^a_J$ can be expressed modulo the ideal $I^a$ 
as a homogeneous  polynomial in coordinates $X^a_{1,\dots \hat{l}\dots, d,j}$ ($l\le d<j$).

We show that $F\in I^a$ in the following way: 
first assume that $F$ depends only on the variables $X^a_{i_1,\dots,i_d}$
with a fixed $d$. Then $F$ vanishes on a dense open cell of the Grassmanian 
$Gr(d,n)$ and hence
the polynomial $F$ belongs to $I$. Since $R^{k;a}_{L,J}=R^{k}_{L,J}$ (up to a superscript $a$)
if lengths of $L$ and $J$ coincide, we obtain $F\in I^a$. 
Now consider the general case, i.e. assume that $F$ depends on
at least two groups of coordinates (with different $d$). Let
$r$ be the maximal number such that for all $m<r$ $F$ does not depend
on the variables $X^a_{i_1,\dots, i_{d_m}}$. 
We then show that
there exists a polynomial $G$ such that $G-F\in I^a$ and
$G$ is independent of the variables $X^a_{i_1,\dots, i_{d_m}}$ with $m\le r$. 

Without loss of generality we assume that $r=1$. First, there exists 
a number $N$ such that the polynomial $\bar F=(X^a_{1,\dots,d_1})^N F$ can be expressed
modulo $I^a$ as a polynomial independent of $X^a_{j_1,\dots,j_d}$ whenever
$\deg (j_1,\dots,j_d)>1$. Note also that $\bar F$ still vanishes on $U$.
Recall the expression \eqref{Xc} for  the degree one coordinates:
\[
X^a_{1,\dots\hat{l}\dots,d_1,m}({\bf c})=(-1)^{d_1-l} c_{l,m-1}.
\]
From this formula and since $\bar F$ vanishes on the whole $U$, $\bar F$ does not depend 
on $X^a_{1,\dots\hat{l}\dots,d_1,m}$ with $m\le d_2$ (because numbers $c_{l,m-1}$
with $m\le d_2$ can come nowhere, but from $X^a_{1,\dots\hat{l}\dots,d_1,m}({\bf c}))$.
For the rest $X^a_{1,\dots\hat{l}\dots,d_1,m}$ we have the following relation in $I^a$:
\[
X^a_{1,\dots,d_2}X^a_{1,\dots\hat{l}\dots,d_1,m} - 
(-1)^{d_1 + d_2} X^a_{1,\dots\hat{l},\dots,d_2,m}X^a_{1,\dots,d_1}=0.
\]
Therefore, there exist  numbers $M_1$ and $M_2$ and a polynomial $G$ such that 
\[
(X^a_{1,\dots,d_2})^{M_1}\bar F - (X^a_{1,\dots,d_1})^{M_2} G\in I^a
\] 
and $G$ does not depend on variables $X^a_{i_1,\dots,i_{d_1}}$.
Note that since $\bar F$ vanishes on $U$ and $X^a_{1,\dots,d}$ equals to $1$ on $U$,
the polynomial $G$ vanishes on $U$ as well. 
By induction on $s$  we can assume $G\in I^a$. Therefore,
\[
(X^a_{1,\dots,d_2})^{M_1}\bar F=(X^a_{1,\dots,d_2})^{M_1}(X^a_{1,\dots,d_1})^N F\in I^a.
\] 
Since $I^a$ is prime (see Corollary \ref{prime}), we arrive at $F\in I^a$.
\end{proof}

\subsection{Flatness}
We close with the proof of the flatness of the degeneration $\Fl\to \Fl^a$.
Let $t$ be a variable. We define an algebra $Q^t$ over the ring $\bC[t]$ as a
quotient of the polynomial ring $\bC[t][X_{i_1,\dots,i_d}]$, $d=d_1,\dots,d_s$ by the ideal
$I^t$ generated by quadratic relations $R^{k;t}_{L,J}$. These relations are $t$-deformations
of the relations $R^{k}_{L,J}$. Namely, let $R^{k}_{L,J}=\sum_i X_{L^{(i)}}X_{J^{(i)}}$.
Then
\[
R^{k;t}_{L,J}=t^{-\min_i (\deg L^{(i)}+\deg J^{(i)})}\sum_i t^{\deg L^{(i)}+\deg J^{(i)}}
X_{L^{(i)}}X_{J^{(i)}}.
\] 
The following lemma describes the fibers of $Q^t$.
\begin{lem}
$Q^t/(t)\simeq Q^a$, $Q^t/(t-u)\simeq Q$ for $u\ne 0$.
\end{lem}
\begin{proof}
Straightforward.
\end{proof}

\begin{prop}
$Q^t$ is $\bC[t]$ free.
\end{prop}
\begin{proof}
We prove that the elements $X_T$, $T\in ST^a_\la$, $\la\in P^+(d_1,\dots,d_s)$
form a $\bC[t]$ basis of $Q^t$. First, let us prove that the elements $X_T$ are linearly independent.
Assume $$\sum_T m_T X_T=\sum_i R^{k_i;t}_{L^{(i)},J^{(i)}}p_i(t),\ m_T, p_i\in\bC[t].$$
There exists a number $u\in\bC$ such that $m_T(u)\ne 0$ if $m_T\ne 0$. Now 
Proposition \ref{nb} gives linear independence. The proof of the spanning property
is analogous to the proof of Proposition \ref{nb}.  
\end{proof}

\section*{Acknowledgments}
We are grateful to I.Arzhantsev, M.Finkelberg and A.Kuznetsov for useful discussions and
to I.Arzhantsev for useful remarks on the earlier version of the paper.
This work was partially supported
by the Russian President Grant MK-281.2009.1,  RFBR Grants 09-01-00058,
07-02-00799 and NSh-3472.2008.2, by Pierre Deligne fund
based on his 2004 Balzan prize in mathematics and by EADS foundation chair in mathematics.

\end{document}